\documentclass[a4paper,11pt]{article}

\usepackage[utf8]{inputenc}
\usepackage[sumlimits,intlimits]{amsmath}
\usepackage{mathrsfs,amssymb,amsthm,mathtools,bbm}

\usepackage[T1]{fontenc}
\usepackage{lmodern,hyperref}
\numberwithin{equation}{section}

\usepackage{geometry}
\usepackage{fancyhdr,lastpage}
\usepackage[nofancy]{svninfo}
\newtheorem{theorem}{Theorem}[section]
\newtheorem{lemma}[theorem]{Lemma}
\theoremstyle{definition}
\makeatletter
\newtheorem{przykl@d}[theorem]{Example}
\newenvironment{example}{\begin{przykl@d}}{\qed\end{przykl@d}}
\makeatother
\theoremstyle{remark}

\newtheorem{assumption}[theorem]{Assumption}

\newcommand{\field}[1]{\mathbbm{#1}}
\newcommand{\fC}{\field{C}}% symbol of set of complex numbers C
\newcommand{\fN}{\field{N}}% symbol of set of natural numbers N={1,2,...}
\newcommand{\oper}[1]{\mathbb{#1}}
\newcommand{\bP}{\oper P}

\newcommand{\bL}{\oper L}

\DeclareMathOperator{\opID}{\oper I}
\DeclareMathOperator{\opE}{\oper E}
\DeclareMathOperator{\opDelta}{\Delta}

\newcommand{\transf}[3]{#1_{#2}^{(#3)}}
\newcommand{\Tnw}[2]{\opT{#2}{#1}}
\newcommand{\Snw}{\mathscr{Q}}
\newcommand{\optr}[2]{{\oper #1}^{(#2)}}
\newcommand{\opL}[1]{\bL^{(#1)}}
\newcommand{\opP}[1]{\bP^{(#1)}}
\newcommand{\opT}[2]{{\Snw^{(#1)}_{#2}}}
\newcommand{\opN}[2]{{N^{(#1)}_{#2}}}
\newcommand{\opD}[2]{{D^{(#1)}_{#2}}}
\newcommand{\abs}[1]{\left\lvert#1\right\rvert}
\newcommand{\acc}{\mathtt{acc}}
\newcommand{\bigO}[1]{\mathcal{O}\!\left(#1\right)}

\newcommand{\ftime}[1]{\textit{#1}}

\newcommand{\hyper}[5]{\,\raisebox{-2pt}{\mbox{${}_{#1}{\text{\rm{\Large F}}}_{\!#2}$}}\!\left(\!\!\begin{array}{c|}{#3}\\[1ex]{#4} \end{array}\,\,\,{#5}\right)}
\newcommand{\texthyper}[5]{{}_{#1}{\mathrm F}_{\!#2}(#3;\,{#4}\,|\,{#5})}

\let\alphaold\alpha
\let\betaold\beta
\let\gammaold\gamma
\newcommand{\mult}[1]{%
	{% - this brace has a great importance, since there are \renecommands inside
	\renewcommand{\beta}{\boldsymbol\betaold}%
	\renewcommand{\alpha}{\boldsymbol\alphaold}%
	\renewcommand{\gamma}{\boldsymbol\gammaold}%
	{\langle #1\rangle}}%
}
\newcommand{\balpha}{\boldsymbol\alpha}
\newcommand{\bbeta}{\boldsymbol\beta}
\newcommand{\wideversion}[2]{#1}%    for wide page styles
\title{New properties of a~certain method of~summation of~generalized hypergeometric series}
\author{%
  Rafa\l{} Nowak%
  \footnote{Institute of Computer Science, University of Wrocław, Poland, e-mail: \href{mailto:rafal.nowak@cs.uni.wroc.pl}{rafal.nowak@cs.uni.wroc.pl}, corresponding author}%
  \and
  Pawe\l{} Wo\'zny%
  \footnote{Institute of Computer Science, University of Wrocław, Poland, e-mail: \href{mailto:Pawel.Wozny@ii.uni.wroc.pl}{Pawel.Wozny@ii.uni.wroc.pl}}%
}
\date{%
  \svnInfoLongDate%
  \footnote{Rev.\ \svnInfoRevision, \svnInfoLongDate, \svnInfoTime}%
}

\begin{document}
\svnInfo $Id: convacc2.tex 2355 2016-08-30 13:00:23Z rno $
\maketitle
\begin{abstract}
  In a~recent paper (Appl.~Math.~Comput.~215, 1622--1645, 2009), the authors 
  proposed a~method of summation of some slowly convergent series. The purpose 
  of this note is to give more theoretical analysis for this transformation, including the
  convergence acceleration theorem in the case of summation of generalized 
  hypergeometric series. Some new theoretical results and illustrative 
  numerical examples are given.
\end{abstract}

%BEGINOFTEXT - do not remove this comment

\section{Introduction}
The \ftime{generalized hypergeometric series}
$
  \texthyper{p}{q}%
  {\alpha_1,\ldots,\alpha_{p}}
  {\beta_1,\ldots,\beta_q}{x}
$
is defined by
\begin{equation}\label{E:hser_pre}
  \hyper{p}{q}%
  {\alpha_1,\alpha_2,\ldots,\alpha_{p}}
  {\beta_1,\beta_2,\ldots,\beta_q}{x}
  \coloneqq
  \sum_{n=0}^\infty
  \frac{(\alpha_1)_n(\alpha_2)_n\cdots(\alpha_p)_n}
      {(\beta_1)_n(\beta_2)_n\cdots(\beta_q)_n}\cdot\frac{x^n}{n!}
\end{equation}
for the given non-negative integer numbers $p$ and $q$, complex parameters
$\alpha_1,\alpha_2,\ldots,\alpha_p$,\linebreak
$\beta_1,\beta_2,\ldots,\beta_q$ and $x$
(see, e.g.,~\cite[\S2.1]{AAR}), where
$(z)_0 \coloneqq 1$, $(z)_n \coloneqq
z(z+1)(z+2)\cdots(z+n-1)$ $(n\geq 1)$ is the \ftime{Pochhammer symbol}.
Using series~\eqref{E:hser_pre}, one can represent most
of~elementary and special functions, as well as constants
that arise in mathematics and physics (see, e.g.,~\cite{Olver2010}).
Hence, this class of series has many applications in the approximation
theory and numerical analysis.

Evaluating a~generalized hypergeometric series of the form~\eqref{E:hser_pre}
can be a~very challenging numerical problem.
First, its~convergence or divergence depends mainly on the values of the
numbers $p$ and $q$. Second, the variety of the parameters~$\alpha_j, \beta_j, x$ leads
to different types of the convergence or divergence of the series.

It is worth mentioning that the \ftime{sequence transformations} may be very
useful numerical tools for the summation of the series given by~\eqref{E:hser_pre}.
In the most recent edition of the book~\textit{Numerical Recipes} by Press et al., some classic sequence
transformations are described involving summation of convergent or divergent series; see~\cite[\S 5.3]{NumRec}.

Probably the best known class of sequence transformations are Pad\'e approximants which
deal with partial sums of the power series and transform them to a~double indexed sequence
of rational approximants. For further information on Pad\'e approximants, we refer to
the books by Baker~\cite{Baker75} or by Baker and Graves-Morris~\cite{BakerGravesMorris}.

As an alternative to the theory of Pad\'e approximants,
one can use the sequence transformations, for which the fundamentals were
given by Wimp \cite{Wimp}, and the extrapolation methods 
described by~Bezinski and Redivo Zaglia in~\cite{BrezinskiZaglia91} or~Sidi~\cite{Sidi03}. 
Most of the classic algorithms were also well summarized by Weniger in
the~report~\cite{Weniger} and Homeier in~\cite{Homeier}.
Undoubtedly these sequence transformations have many common properties with Pad\'e approximants,
but in the case of generalized hypergeometric series they can be much more powerful tools.
It is worth mentioning that evaluation of special functions with the help of
convergence acceleration techniques is also considered in the book by~Olver et al.~\cite[\S3.9]{Olver2010}.

In general, the success of convergence acceleration of the series~$\sum_{n=0}^\infty a_n$
very often depends on the behavior of the sequence of its \ftime{partial
sums} $s_n \coloneqq \sum_{j=0}^{n-1} a_j$ and~\ftime{remainders} $r_n \coloneqq
\sum_{j=0}^\infty a_{n+j}$. Supposing $s_n\to s$, one has $s=s_n+r_n$. We say that $\{s_n\}$ 
converges \ftime{linearly}, if $\lim_{n\to\infty}
(s_{n+1}-s)/(s_n-s)\eqqcolon\varsigma$ and $0<\abs\varsigma<1$. In the case of
$\varsigma=1$, we have \ftime{logarithmic} convergence, which is usually the
most difficult to accelerate.

If~$p-q=1$ then the series~\eqref{E:hser_pre} may be extremely slowly
convergent and practically unusable.
Therefore, the methods of convergence acceleration of~such series are particularly important.
For instance, in order to use any of so-called \ftime{Levin-type} sequence~transformations,
one should first provide good estimates $\omega_n$ of the remainders $r_n$; see, e.g., \cite{Homeier}.
This can be achieved with the help of recent results given by Willis
in~\cite{Willis2012}. Namely, he analyzed the following asymptotic relation for
the remainders of the~series~\eqref{E:hser_pre} with $p-q=1$:
\begin{equation}
  s-s_n = r_n \sim \mu x^n n^\lambda \sum_{k=0}^\infty \frac{c_k}{n^k}.
\end{equation}
He derived the recurrence relation for the coefficients $c_k$, which gives the
ability to approximate the remainders $r_n$ up to the desired order.
Willis combined this asymptotic property with a~classic technique of
extrapolation and obtained a~numerical algorithm of computing 
2-dimensional array containing approximations of the limit of the series.
Namely, the truncated estimate 
\[ \transf\omega nm \coloneqq x^n n^\lambda \sum_{k=0}^\infty \frac{c_k}{n^k}, \qquad m\in\fN, \]
is such that
\[ s = s_n + \mu \, \transf\omega nm + \bigO{x^n n^{\lambda-m}} \]
and thus the new approximation defined by the weighted average
\[
  \transf snm \coloneqq \frac{\transf \omega{n+1}m s_n - \transf\omega nm s_{n+1}}{\transf\omega{n+1}m-\transf\omega nm}
\]
is a~better approximation of the partial sum $s_n$ in the sense that the sequence $\transf snm$
converges to the limit $s$ faster than the sequence of partial sums $s_n$.
It is worth remarking that Willis' method is very efficient at the branch point $x=1$.
However, it does not seem to provide good numerical properties, if $x\approx 1$.

In this paper, we continue the analysis of $\Snw$ transformation, introduced by us in~\cite{WoznyNowak},
applied to the summation of~generalized hypergeometric series~\eqref{E:hser_pre} with~$p-q=1$.
Following the notation given in~\cite{WoznyNowak}, we consider the series
\begin{equation}\label{E:hser}
  \sum_{n=0}^\infty a_n \qquad\text{with}\qquad
    a_n \coloneqq %\frac{\mult{\alpha}_n}{\mult{\beta}_n}\,x^n =
    \frac{(\alpha_1)_n(\alpha_2)_n\cdots(\alpha_p)_n}
    {(\beta_1)_n(\beta_2)_n\cdots(\beta_p)_n}\,x^n
    \quad
    (x\in\fC),
\end{equation}
where $\boldsymbol\alpha\coloneqq (\alpha_1,\alpha_2,\ldots,\alpha_p)$,
$\boldsymbol\beta\coloneqq (\beta_1,\beta_2,\ldots,\beta_p)$
are vectors of~complex parameters.
Notice that the~series~\eqref{E:hser} corresponds to the following
generalized hypergeometric series
\begin{equation*}%\label{E:hser}
    \hyper{p+1}{p}
    {\alpha_1,\alpha_2,\ldots,\alpha_p,1}
    {\beta_1,\beta_2,\ldots,\beta_p}
    {x}
    \!,
\end{equation*}
which is the case of~the series~\eqref{E:hser_pre} with $p-q=1$. Let us remark that $\Snw$ transformation
can be applied to the series~\eqref{E:hser_pre} also in the case with none of the
upper parameters being equal to $1$, since one can always use the following obvious relation
\[
 \hyper{r}{s}{\alpha_1,\alpha_2,\ldots,\alpha_r}{\beta_1,\beta_2,\ldots,\beta_s}{x} = 
 \hyper{r+1}{s+1}{\alpha_1,\alpha_2,\ldots,\alpha_r,1}{\beta_1,\beta_2,\ldots,\beta_s,1}{x}.
\]

The main purpose of this paper is to give more theoretical properties of the $\Snw$ transformation.
For a~detailed comparison with other methods of convergence 
acceleration, we refer to~\cite{WoznyNowak} and~\cite{Wozny2010}, where the bunch of numerical examples
involving many classic and recent sequence transformations, such as 
Aitken's iterated $\Delta^2$ process \cite{Aitken26},
Wynn's $\varepsilon$-algorithm \cite{Wynn56}, 
$t$ and $u$ variants of Levin \cite{Levin73} and Weniger \cite{Weniger92} transformations, 
Homeier's transformations \cite{Hom94}, 
method proposed by Lewanowicz and Paszkowski \cite{LewanowiczPaszkowski}, 
method proposed by {\v{C}{\'i}{\v z}ek} et~al. \cite{CZS} and 
the method of Paszkowski \cite{Paszkowski08}, were given;
see also the Brezinski's review of convergence acceleration techniques \cite{Brezinski00}.

The convergence of~the series~\eqref{E:hser} depends on the variety of
parameters $\balpha$, $\bbeta$ and the complex number $x$.
If $\abs x<1$, then the series converges absolutely.
On the unit circle, the convergence is more subtle and depends on the real
part of the parameter
$\sigma \coloneqq 1+\sum_{i=1}^p \alpha_i - \sum_{i=1}^p \beta_i$.
Namely, the series~\eqref{E:hser} converges at $x=1$, if~$\Re\sigma < 0$,
as well as for $\abs x=1$ ($x\neq 1$), if $\Re\sigma < 1$.
Many other mathematical properties of the generalized hypergeometric
series ${}_{p+1}F_p$ can be found in such classic sources as
\cite{Abramowitz}, \cite{AAR}, \cite{Magnus} or~\cite{Slater}.
It is also worth mentioning the recent research
of Miller and Paris \cite{MillerParis2011,MillerParis2012,MillerParis2013},
Rathie and Paris \cite{RathieParis2013} and Kim et al.~\cite{KimRathieParis2014},
where certain transformations were proposed to simplify the so-called \textit{order} of generalized hypergeometric
series. Several refinements of such a~class of series can be also found in the very recent work of Wang \cite{Wang2016}.

In~\cite{WoznyNowak}, the authors introduced the technique
of~summation of some slowly convergent series, which appeared to be very
efficient also in the case of the generalized hypergeometric series.
The proposed $\Snw^{(m)}$ transformation is %a~Levin-type transformation,
defined by a~certain linear difference operator $\opL{m}_n$ in the 
following way:
\begin{equation}\label{E:Tnw:L}
  \Tnw nm \coloneqq \frac{\opL{m}_n(s_n)}{\opL{m}_n(1)}, \qquad m\in\fN.
\end{equation}
Here, and in the sequel, every difference operator acts upon $n$ and not upon $m$.
The meaning of the linear operator $\opL{m}_n$ is that it annihilates 
a~finite part of the remainder $r_n$. More precisely, it is required that
\begin{equation}
  \label{E:Lm:sat}
  \opL{m}_n\left(a_n+a_{n+1}+\ldots+a_{n+m-1}\right)= 0.
\end{equation}
% Thus, the quantities $\Tnw nm$ are supposed to be better and better approximations
% of the limit of the series.
In a~general case, for~arbitrary sequence $a_n$, the computation of~the
quantities $\Tnw nm$ is rather complicated and the numerical version of the
algorithm is recommended; see~\cite{Wozny2010}.
However, in the case of the series~\eqref{E:hser},
one can compute the quantities $\Tnw nm$ in a~very efficient way, i.e., using 
the~algorithm involving certain recurrence formulas for numerators and 
denominators in~\eqref{E:Tnw:L}; cf.~\cite[Alg.~1, Thm.~2]{WoznyNowak}. 

For~convenience of~reference, we briefly summarize the main results 
from~\cite{WoznyNowak} in the case of the series~\eqref{E:hser}.
For a~given vector
$\boldsymbol\gamma=(\gamma_1,\gamma_2,\ldots,\gamma_p)$,
we use the~following shorthand notation
\begin{equation}
  \label{E:mult}
  \mult{\gamma}_n\coloneqq\prod_{j=1}^p (\gamma_j)_n,\qquad n\in\fN\cup\{0\},
\end{equation}
and write $a_n = \mult{\alpha}_n/\mult{\beta}_n\,x^n$.
The operators $\opL{m}_n$ satisfying \eqref{E:Lm:sat} can be written in the form
\begin{equation}
  \label{E:Lm:p1Fp}
  \opL{m}_n = 
  \opDelta^{mp}\left(\frac{\mult{\beta}_{n+m-1}}{\mult{\alpha}_{n}\,x^{n}} 
  \opID\right),\qquad{m\in\fN};
\end{equation}
see~\cite[Eq.~(3.4)]{WoznyNowak}.
Here, the \ftime{forward difference operator} $\opDelta$ and 
the~\ftime{identity operator} $\opID$ are defined according to
$\opDelta z_n \coloneqq z_{n+1}-z_n$ and $\opID z_n \coloneqq z_n$, respectively;
higher powers  of~the~operator $\opDelta$ are defined recursively,
i.e., $\opDelta^0 z_n \coloneqq z_n$ and $\opDelta^m z_n \coloneqq \opDelta( \opDelta^{m-1} z_n), m\in\fN$.
From a~computational point of view, it is worth noting that operators
$\opL{m}_n$ can be written also in the factored form
$\opL{m}_n = \opP{m}_n\opP{m-1}_n\cdots\opP{1}_n$, where the operators
$\opP m_n$ are defined by
\begin{subequations}
  \label{E:Pm:q1Fq}
  \begin{align}
    \opP{1}_n &\coloneqq 
    \opDelta^p\left(\frac{\mult{\beta}_n}{\mult{\alpha}_n}\,x^{-n} \opID\right),
    \\
    \label{E:Pm:q1Fq:b}
    \opP{m}_n &\coloneqq \sum_{j=0}^{p} \binom{mp}{j}
    \left[\opDelta^j\prod_{j=1}^{p}
    \left({\beta_j+n+m(p+1)-j-2}\right)
    \right]\opDelta^{p-j},\wideversion{\qquad}{\quad} m\geq 2.
  \end{align}
\end{subequations}
Thus, the quantities $\Tnw nm$ can be computed using the following recursive 
scheme (see~\cite[Alg.~1]{WoznyNowak}):
\begin{subequations}\label{E:alg}
\begin{alignat}{3}
  \opN 0n &\coloneqq s_n,&\qquad \opD 0n &\coloneqq 1,\\
  \opN mn &\coloneqq \opP m_n(\opN{m-1}{n}),&\qquad 
  \opD mn &\coloneqq \opP m_n(\opD{m-1}{n}),&\qquad m&\geq 1,\\
  \Tnw{n}{m} &= \frac{\opN{m}{n}}{\opD mn}.
\end{alignat}
\end{subequations}
From eqs.~\eqref{E:Pm:q1Fq} and~\eqref{E:alg},
one may conclude that
\begin{equation}\label{E:Tnw:transformation}
 \Tnw nm = \Tnw nm( s_n, s_{n+1}, \ldots, s_{n+\ell(m)} ),
\end{equation}
where $\ell(m) = mp$, which means that $\Snw^{(m)}$ transforms the sequence $\{s_n\}_{n=0}^\infty$
to the sequence $\{\Tnw nm\}_{n=0}^\infty$ whose $n$-th element depends on $s_n, s_{n+1}, \ldots, s_{n+mp}$.
Let us remark that Levin-type sequence transformations produce the double indexed quantities
\[ \transf{\mathcal L}nm \coloneqq \transf{\mathcal L}nm( \{\omega_n\}, s_n, s_{n+1}, \ldots, s_{n+\ell(m)} ), \]
depending both on the partial sums and on the sequence of remainder estimates~$\{\omega_n\}$ (see, e.g., \cite{Homeier}),
with such relationships as $\ell(m) = m$, $\ell(m) = m+1$,
$\ell(m) = 2m$ or $\ell(m) = 3m$; see, e.g., \cite[\S2.7]{Weniger}.
For example, the 
% However, it is important to distinguish the preliminary work, given in eqs.~\eqref{E:Pm:q1Fq},
% in order to compare the equation~\eqref{E:Tnw:transformation} with Levin-type methods producing the double
% indexed quantities
% \[ \transf{\mathcal L}nm \coloneqq \transf{\mathcal L}nm( \{\omega_n\}, s_n, s_{n+1}, \ldots, s_{n+\ell(m)} ), \]
% depending both on the partial sums and on the sequence of remainder estimates~$\{\omega_n\}$; see, e.g., \cite{Homeier}.
classic variants of Levin transformation have $\ell(m) = m$ and involve the following choices of remainder estimates: 
$\omega_n = a_n$, $\omega_n = a_{n+1}$, $\omega_n = (n+1) a_n$
or $\omega_n = \frac{a_n a_{n+1}}{a_n - a_{n+1}}$; see~the paper of Levin \cite{Levin73} 
and work of Smith and Ford \cite{SmithFord79}.
The advantage of $\Snw^{(m)}$ transformation is that the information about remainder estimates $\omega_n$
is a~priori hidden in the analytic form of the~operators~$\opP m_n$, given by the explicit formulas~\eqref{E:Pm:q1Fq}.
% This preprocessed analytical work cannot be compared with any of~the classic choices of the estimates
% $\omega_n$ in~Levin-type methods. 
It should be remarked that the~operators~\eqref{E:Pm:q1Fq} seem
to provide the~transformation $\Snw^{(m)}$ which is a~very powerful numerical tool for the summation
of the generalized hypergeometric series \eqref{E:hser_pre} with $p-q=1$.

Some theoretical properties of~$\Snw^{(m)}$ transformation were also given
in~\cite{WoznyNowak}. 
In the case of the series~\eqref{E:hser}, we can summarize them as follows.

If $p=1$, then $\Snw^{(m)}$ transformation is~equivalent to Wynn's
$\varepsilon$ algorithm \cite{Wynn56} in the sense that
$\Tnw{n}{m} = \varepsilon_{2m}^{(n)}$, which follows from the general property
given in~\cite[\S2.3]{WoznyNowak}.
It is worth mentioning that the explicit formula for $\varepsilon_{2m}^{(n)}$
in the case of $\texthyper{2}{1}{\alpha,1}{\beta}{x}$ has already been given
by Sidi in~\cite[Ex.~2]{Sidi81} (see also \cite[\S17.3]{Sidi03}).
The mentioned relation between $\Snw^{(m)}$ and $\varepsilon$ transformations
does not hold for $p>1$.

Supposing that the series~\eqref{E:hser} is convergent, we also know that~$\Snw^{(m)}$ transformation is \ftime{regular}
for all $m\in\fN$, if~$x\neq 1$, i.e.,
\[
\lim_{n\to\infty}\Tnw{n}{m}\left(s_n\right)
=
\hyper{p+1}{p}{\alpha_1,\alpha_2,\ldots,\alpha_p,1}
{\beta_1,\beta_2,\ldots,\beta_p}{x};
\]
see~\cite[Thm.~5]{WoznyNowak}.
What is more, it possesses the following asymptotic behavior:
\[
\hyper{p+1}{p}{\alpha_1,\alpha_2,\ldots,\alpha_p,1}
{\beta_1,\beta_2,\ldots,\beta_p}{x}-
\Tnw{n}{m}\left(s_n\right)
=\bigO{x^{{n+m(p+1)}}}, \qquad x\rightarrow 0;
\]
see~\cite[Thm.~6]{WoznyNowak}.

Moreover, a~lot~of numerical tests show that the sequence $\Tnw nm$ not only 
converges to the limit of the series but also converges much faster for bigger 
and bigger values of~$m$.

The~paper is organized as follows. In Section \ref{S:new}, we give some new 
properties of 
$\Snw^{(m)}$ transformation in the case of the series~\eqref{E:hser} including the 
main result, which is the convergence acceleration theorem.
Later, in~Section~\ref{S:experiments}, we give some
numerical examples and, in Section~\ref{S:problems}, discuss further problems concerning 
theoretical properties of $\Snw^{(m)}$ transformation.

\section{Main result}\label{S:new}
Let us consider~the~transformation $\Tnw{}m$, defined~in~\eqref{E:Tnw:L},
in~the case of the series~\eqref{E:hser}.
Following~\cite{WoznyNowak}, we define the functions $\transf{\lambda}{j}{m}(n)
\equiv \transf{\lambda}{j}{m}(n,p,x,\boldsymbol\alpha,\boldsymbol\beta)$ by
\begin{equation*}
  \transf{\lambda}{j}{m}(n)
  \coloneqq
  \wideversion{%
    \frac{\mult{\alpha}_{n+mp}\,x^{n+mp}}{\mult{\beta}_{n+m-1}}
    \left[ (-1)^{mp-j}\binom{mp}{j}
          \frac{\mult{\beta}_{n+j+m-1}}{\mult{\alpha}_{n+j}\,x^{n+j}}
          \right], \quad j=0,1,\ldots,mp,
  }{%
    (-1)^{mp-j}\binom{mp}{j}
    \frac{\mult{\alpha}_{n+mp}\,x^{n+mp}}{\mult{\beta}_{n+m-1}}
    \frac{\mult{\beta}_{n+j+m-1}}{\mult{\alpha}_{n+j}\,x^{n+j}},\;
    j=0,1,\ldots,mp,
  }
\end{equation*}
and thus, using~\eqref{E:Lm:p1Fp}, we can write the 
transformation~$\Tnw{}m$ as follows:
\begin{equation}
  \label{E:Q:lambda}
  \Tnw{n}{m} = \frac{\optr Lm(s_n)}{\optr Lm(1)}
%   = \frac{\displaystyle \frac{\mult{\alpha}_{n+mp}\,x^{n+mp}}{\mult{\beta}_{n+m-1}}\cdot\optr Lm(s_n)}{\displaystyle \frac{\mult{\alpha}_{n+mp}\,x^{n+mp}}{\mult{\beta}_{n+m-1}}\cdot\optr Lm(1)}
  = \frac{\displaystyle \sum_{j=0}^{mp}{\transf{\lambda}{j}{m}(n)}\, s_{n+j}}{\displaystyle \sum_{j=0}^{mp}{\transf{\lambda}{j}{m}(n)}}.
\end{equation}
Since
\begin{equation}\label{E:lambda}
\transf{\lambda}{j}{m}(n) =
\binom{mp}{j}(-x)^{mp-j}\,\prod_{i=1}^p
\left[ (\alpha_i+n+j)_{mp-j}(\beta_i+n+m-1)_j \right]
\end{equation}
(see~\cite[Eq.~(3.11)]{WoznyNowak}), the quantity $\Tnw nm$ is a~linear 
combination of the quantities $s_n$, $s_{n+1}$, \ldots, $s_{n+mp}$ 
with~coefficients~$\transf\lambda jm(n)$
being polynomials of degree $mp^2$ in~$n$.
In the~lemma~below, we~express the~element $\Tnw nm$ in terms of~$s_n$
and~$a_n, a_{n+1}, \ldots, a_{n+mp-1}$.
In the sequel, we use the following polynomials in~$n$:
\begin{equation}\label{E:M}
{M^{(m)}(n)} \equiv M^{(m)}_0(n), \qquad
{M^{(m)}_k(n)} \coloneqq \sum_{j=k}^{mp}\transf{\lambda}{j}{m}(n), \qquad
k=0,1,\ldots,mp.
\end{equation}
\begin{lemma}\label{L:Q:a}
  The quantity $\Tnw nm$ can be written as the following linear combination 
  involving the partial sum $s_n$ and the terms
  $a_n$, $a_{n+1}$, \ldots, $a_{n+mp-1}$:
  \begin{equation}
    \Tnw{n}{m} = s_n+\sum_{k=0}^{mp-1} \frac{M^{(m)}_{k+1}(n)}{M^{(m)}(n)} \, a_{n+k},
    \qquad m,n\in\fN.
  \end{equation}
\end{lemma}
\begin{proof}
  First, let us observe that
  \( \Tnw{n}{m} = \sum_{j=0}^{mp}\transf{\lambda}{j}{m}(n) s_{n+j}/M^{(m)}(n). \)
  Second, we have
  \begin{multline*}%\label{Qn_sn}
    \Tnw{n}{m}
%     =\frac{\displaystyle \sum_{j=0}^{mp}\transf{\lambda}{j}{m}(n)\, s_{n+j}}{M^{(m)}(n)}
    =\frac{\displaystyle \sum_{j=0}^{mp}\transf{\lambda}{j}{m}(n)\left(s_n+a_n+a_{n+1}+\ldots+a_{n+j-1}\right)}{M^{(m)}(n)}\\
    =\frac{\displaystyle \sum_{j=0}^{mp}\transf{\lambda}{j}{m}(n) s_n+\sum_{k=0}^{mp-1} \left(\sum_{j=k+1}^{mp}\transf{\lambda}{j}{m}(n)\right)a_{n+k}}{M^{(m)}(n)}
    =s_n+\sum_{k=0}^{mp-1} \frac{M^{(m)}_{k+1}(n)}{{M^{(m)}(n)}} a_{n+k}.
  \end{multline*}
\end{proof}

Now, we are going to give some theoretical results about the convergence
acceleration performed by $\Snw$ transformation.
The first theorem gives the necessary and sufficient condition for the
convergence acceleration. Next, we show that this condition holds under 
a~certain assumption which is discussed later --- in Subsection~\ref{comments}.
The statement concerning the convergence acceleration of the linearly 
convergent series~\eqref{E:hser} is given thereafter.
We also analyze the convergence at the points $x=\pm 1$.
\begin{theorem}\label{T}
  Consider the series~\eqref{E:hser} with partials sums $s_n$ converging to $s$.
  The transformation $\Snw^{(m)}$ accelerates the convergence of $\{s_n\}$, 
  i.e.,
  \begin{equation}
    \lim_{n\to\infty}\frac{\Tnw nm - s}{s_n-s} = 0, \qquad m\in\fN,
  \end{equation}
  if and only if
  \begin{equation}\label{E:T:cond}
    \lim_{n\to\infty} \sum_{k=0}^{mp-1} \frac{M^{(m)}_{k+1}(n)}{M^{(m)}(n)}
    \cdot \frac{a_{n+k}}{r_n} = 1,
  \end{equation}
  where $M_k(n)$ and $M(n)$ are defined by~\eqref{E:M}.
\end{theorem}
\begin{proof}
  The theorem follows immediately from Lemma~\ref{L:Q:a}, since
  \begin{equation*}
    \lim_{n\to\infty}\frac{\Tnw nm - s}{s_n-s} = 1 -
    \lim_{n\to\infty}\left(\frac{1}{r_n} \sum_{k=0}^{mp-1}
\frac{M^{(m)}_{k+1}(n)}{M^{(m)}(n)} a_{n+k}\right).
  \end{equation*}
\end{proof}

In the next theorem, we show that the condition~\eqref{E:T:cond} is fulfilled
for a~certain class of convergent hypergeometric series~\eqref{E:hser}.
Let us remark that the further properties hold upon the following assumption.
\begin{assumption}\label{asmpt}
  We assume that the remainders of the~series~\eqref{E:hser} satisfy
  \begin{equation*}%\label{E:asmpt}
    \lim_{n\to\infty} \frac{r_{n+1}}{r_n} = x.
  \end{equation*}
\end{assumption}

\begin{theorem}\label{T:M:a}
  If~Assumption~\ref{asmpt} holds and $x\neq 1$, then the condition
  \eqref{E:T:cond} is satisfied for all $m\in\fN$,
  and thus the transformation $\Snw^{(m)}$ accelerates the convergence of $\{s_n\}$.
\end{theorem}
\begin{proof}
  Let $\transf cjm$ be the~coefficient for the term $n^{mp^2}$ 
  in~$\transf\lambda jm(n)$, i.e.,
  \[ \transf cjm \coloneqq \binom{mp}{j}(-x)^{mp-j}, \quad j=0,1,\ldots,mp; \]
  cf.~\eqref{E:lambda}. Using binomial theorem, one can check that the following two relationships hold:
  \begin{equation}
    \label{E:sum_cj}
    \sum_{j=0}^{mp} \transf cjm = (1-x)^{mp}, \qquad
    \sum_{j=0}^{mp} \transf cjm \, x^j = 0.
  \end{equation}
  Therefore, if~$x\neq 1$, the polynomial $\transf M{}m(n)$, given 
  by~\eqref{E:M}, is~exactly of~degree $mp^2$. This yields
  \[ \lim_{n\to\infty} \frac{M^{(m)}_{k+1}(n)}{M^{(m)}(n)} = \sum_{j=k+1}^{mp} c^{(m)}_j \bigg/ \sum_{j=0}^{mp} c^{(m)}_j = (1-x)^{-mp} \sum_{j=k+1}^{mp} c^{(m)}_j. \]
  Using the fact that $a_{n+k} = r_{n+k}-r_{n+k+1}$ and
  Assumption~\ref{asmpt}, we obtain
  \begin{multline*}
    \lim_{n\to\infty} \sum_{k=0}^{mp-1} \frac{M^{(m)}_{k+1}(n)}{M^{(m)}(n)} \cdot \frac{a_{n+k}}{r_n} =
    (1-x)^{-mp} \sum_{k=0}^{mp-1} \sum_{j=k+1}^{mp} c^{(m)}_j (x^{k}-x^{k+1}) \\
    = (1-x)^{-mp} \left( \sum_{j=0}^{mp} c^{(m)}_j - \sum_{j=0}^{mp} c^{(m)}_j x^j \right).
  \end{multline*}
  Now, the result follows from~\eqref{E:sum_cj}.
\end{proof}

\subsection{Some comments}\label{comments}
Let us note that Assumption~\ref{asmpt} holds for each~series~\eqref{E:hser}
with $\abs{x}<1$, which implies the linear convergence.
It follows directly from~\cite[Thm.~1, p.~6]{Wimp} providing that
\[ x = \lim_{n\to\infty} \frac{a_{n+1}}{a_n} = \lim_{n\to\infty}
\frac{r_{n+1}}{r_n}, \]
if $\abs x<1$.
This yields the theoretical explanation of convergence acceleration 
performed by~$\Snw^{(m)}$ transformation in the case of linearly convergent 
series~\eqref{E:hser}.

Assumption~\ref{asmpt} is also satisfied if $x=1$ and all the
terms $a_n$ are real and have the same sign; see~\cite[Thm.~2, p.~26]{Clark68}.
In this case, the series~\eqref{E:hser} converges logarithmically.
It is quite remarkable that the transformation $\Snw^{(m)}$ seems to be 
very efficient also in such a~case.
However, we cannot use Theorem~\ref{T:M:a} for $x=1$ in order to justify 
that~$\Snw^{(m)}$ transformation leads to an~acceleration of the convergence.

Let us remark that Assumption~\ref{asmpt} holds also for a~certain
class of the series~\eqref{E:hser} with $x=-1$. Namely, from
\cite[Thm.~3, p.~26]{Clark68}, we obtain that it is satisfied, if
$\mult{\alpha}_n/\mult{\beta}_n$ decreases monotonically to zero and
\begin{equation}\label{E:alt:cond}
\lim_{n\to\infty} \frac{1+t_{n+1}}{1+t_n}=1,
\end{equation}
where $t_n \coloneqq a_{n+1}/a_n$.

In the next section, we give some comments and numerical examples showing
the consequences of~the given theorems.

\section{Experiments}\label{S:experiments}
Since the main purpose of this paper is to give the theoretical properties
of the $\Snw^{(m)}$ transformation, we refer to~\cite{WoznyNowak} and~\cite{Wozny2010},
for the numerical examples displaying, among other things, the detailed comparison with many
classic methods of convergence acceleration.
However, we would like to give some new examples in order to depict the conclusions
of the theory given in Section~\ref{S:new}.

We use $\Snw^{(m)}$ transformation in order to approximate the sum of the
generalized hypergeometric series
$\texthyper{3}{2}{\alpha_1,\alpha_2,1}{\beta_1,\beta_2}{x}$.
Let us remark that one can compute the array of quantities~$\Tnw nm$ 
using the recurrence formulas~\eqref{E:alg} and replacing the~operators~$\opP m_n$
with the equivalent ones:
\begin{multline}
  \label{eq:Pm:3F2}
  \opP{m}_n \coloneqq x^2(n+2m+\alpha_1-2)_2(n+2m+\alpha_2-2)_2\,\opID\\
  -2x(n+2m+\alpha_1-1)(n+2m+\alpha_2-1)
  [(n+2m+\beta_1-2)(n+2m+\beta_2-2)-m^2+m]
  \opE\\
  +(n+m+\beta_1-1)(n+m+\beta_2-1)(n+3m+\beta_1-2)(n+3m+\beta_2-2)\,\opE^2
\end{multline}
(cf.~\cite[Eq.~(3.8)]{WoznyNowak}),
where the~\ftime{forward shift operators} $\opE$ and $\opE^2$ are 
defined according to $\opE z_n\coloneqq z_{n+1}$, and $\opE^2 z_n \coloneqq z_{n+2}$.

In all the examples, we start with the column $\Tnw n0$ containing some finite
number of partial sums of the series. Next, we use recurrence 
formulas~\eqref{E:alg} and obtain the triangular array of the quantities
$\Tnw nm$.

In the first example, we consider the alternating series
\eqref{E:hser} with $x=-1$. In this case, one can use Theorem~\ref{T:M:a} in order to
explain the convergence acceleration performed by the $\Snw$ transformation.
Indeed, we obtain highly accurate approximations in the triangular array 
$\Tnw nm$.

In the second example, we consider the linearly convergent
series~\eqref{E:hser}. In order to obtain a~very
slow convergence, we take $x\approx 1$.
Again, as a~consequence of Theorem~\ref{T:M:a}, the transformation $\Snw$ 
provides a~quite good approximation of the limit of the series.

In the last example, we consider the series~\eqref{E:hser} with $x=1$,
which means the logarithmic convergence.
One can observe that the $\Snw$ transformation is very efficient, although this good
performance cannot be justified by Theorem~\ref{T:M:a}.

All the numerical experiments were made in
\textsf{Maple{\small\texttrademark}~14} system,
using floating point arithmetic with $32$ decimal digits precision.
Consequently, as~in~\cite{WoznyNowak},
we measure the \ftime{accuracy} of the~approximation $z$ of the sum $s\neq 0$ by
the number of exact significant decimal digits, in the following sense:
\[
  \acc(z)\coloneqq -\log_{10}\abs{\frac{z}{s}-1}.
\]

\begin{example}\label{EX:alt}
  Let us consider the following expansion of the integral involving the
  square root function:
  \begin{equation}\label{E:ex:alt}
    \frac{1}{z}\int_{0}^{z} \sqrt{1+y^\alpha} \mathrm{d}y
    = \hyper{2}{1}{\frac1\alpha, -\frac12}{1+\frac1\alpha}{-z^\alpha}
    \qquad (\alpha>0, \; 0<z\leq 1).
  \end{equation}
  Since the first parameter of the
  hypergeometric series is not equal to $1$, we consider the convergence
  acceleration of the series
  \wideversion{$\texthyper{3}{2}{1/\alpha,-1/2,1}{1+1/\alpha,1}{-z^\alpha}$.}
  {\[\texthyper{3}{2}{1,1/\alpha,-1/2}{1,1+1/\alpha}{-z^\alpha}.\]}
  We put $\alpha = 1/3$ and $z=1$. Thus, we obtain an
  alternating series of the form~\eqref{E:hser} with
  \begin{equation} \label{E:ex:alt:a} 
%    a_{n+1} = \frac{3(2n-1)!!}{(n+4)(2n+2)!!} \, x^n, \qquad x=-1, 
    a_{n} = \frac{3 (-\tfrac12)_n }{(n+3) n!} \, x^n, \qquad x=-1, 
  \end{equation}
  converging to
  $(44\sqrt2-16)/35\approx
%    0.7597979746.%446661366447284277871546199950.
   1.3207256213.
  $
  The convergence is quite slow since straightforward computation gives
  (underline depicts the correct digits):
  \wideversion{\[
  s_{   10} = \underline{1.32}19178336, \quad
  s_{  100} = \underline{1.32072}97959, \quad
  s_{ 1000} = \underline{1.3207256}346, \quad
  s_{10000} = \underline{1.3207256213}.
  \]}{%
  \begin{alignat*}{2}
  s_{   10} &= \underline{1.32}19178336,& \qquad
  s_{  100} &= \underline{1.32072}97959, \\
  s_{ 1000} &= \underline{1.3207256}346,& \qquad
  s_{10000} &= \underline{1.3207256213}.
  \end{alignat*}
  }
  This yields the following number of exact 
  significant decimal digits of the limit of the series:
  \[
  \acc(s_{10}) = 3.0,\quad
  \acc(s_{100}) = 5.5,\quad
  \acc(s_{1000}) = 8.0,\quad
  \acc(s_{10000}) = 10.5.
  \]
  One can check that 
  $t_n \coloneqq a_{n+1}/a_{n} = -[(2n-1)(n+3)]/[(2n+2)(n+4)]$, 
  and thus the equation~\eqref{E:alt:cond} holds.
  Since the fraction in~\eqref{E:ex:alt:a} decreases monotonically to zero, we 
  conclude that Assumption~\ref{asmpt} is satisfied;
  cf.~\cite[Thm.~3, p.~26]{Clark68}.
  From Theorem~\ref{T:M:a}, we obtain that $\Snw^{(m)}$ transformation
  accelerates the convergence of the considered series. Indeed, the accuracy of
  the quantities $\Tnw nm$ increases 
  for bigger and bigger values of $m$; see Table~\ref{TAB:ex:alt}. It is worth 
  noting that the quantity $\Tnw17$ gives about $21$ digits of accuracy, 
  while the partial sums $s_1$, $s_2$, \ldots, $s_{15}$, that it depends on
  (see~Lemma~\ref{L:Q:a}), give less than $4$ digits.
  \begin{table}[htb]
    \caption{\label{TAB:ex:alt}
      Values of $\acc(\Tnw nm)$ for the hypergeometric series~\eqref{E:ex:alt}
      with $\alpha=1/3$ and $z=1$.}
    \center
    \begin{tabular}{c|rrrrrrrr}
      $n\diagdown m$ & \multicolumn{1}{c}{$ 0$}  & \multicolumn{1}{c}{$ 1$}  &
      \multicolumn{1}{c}{$ 2$}  & \multicolumn{1}{c}{$ 3$}  &
      \multicolumn{1}{c}{$ 4$}  & \multicolumn{1}{c}{$ 5$}  &
      \multicolumn{1}{c}{$ 6$}  & \multicolumn{1}{c}{$ 7$} \\\hline
      $ 1$ & $ 0.6$ & $ 3.1$ & $ 6.1$ & $ 9.5$ & $12.6$ & $15.9$ & $19.3$ & 
      $21.7$\\
      $ 2$ & $ 1.4$ & $ 3.8$ & $ 6.9$ & $10.3$ & $13.4$ & $16.9$ & $19.9$\\
      $ 3$ & $ 1.8$ & $ 4.4$ & $ 7.6$ & $11.0$ & $14.2$ & $17.9$ & $20.6$\\
      $ 4$ & $ 2.1$ & $ 4.9$ & $ 8.2$ & $11.6$ & $14.9$ & $18.9$\\
      $ 5$ & $ 2.3$ & $ 5.3$ & $ 8.7$ & $12.1$ & $15.6$ & $20.6$\\
      $ 6$ & $ 2.5$ & $ 5.6$ & $ 9.2$ & $12.7$ & $16.2$\\
      $ 7$ & $ 2.7$ & $ 5.9$ & $ 9.6$ & $13.1$ & $16.7$\\
      $ 8$ & $ 2.8$ & $ 6.2$ & $10.0$ & $13.6$\\
      $ 9$ & $ 2.9$ & $ 6.4$ & $10.4$ & $14.0$\\
      $10$ & $ 3.0$ & $ 6.6$ & $10.7$\\
      $11$ & $ 3.1$ & $ 6.8$ & $11.0$\\
      $12$ & $ 3.2$ & $ 7.0$\\
      $13$ & $ 3.3$ & $ 7.2$\\
      $14$ & $ 3.4$\\
      $15$ & $ 3.5$
    \end{tabular}
  \end{table}
\end{example}

\begin{example}\label{EX:linconv}
  Let us consider the linearly convergent series
  \begin{equation*}%\label{E:ex:linconv}
    \hyper{2}{1}
    {\frac16,\frac13}
    {\frac12}
    {\frac{25}{27}} =
    \hyper{3}{2}
    {\frac16,\frac13,1}
    {\frac12,1}
    {\frac{25}{27}}
    = \frac34\sqrt3 \approx 1.2990381057;
  \end{equation*}
  see \cite{ZuckerJoyce2001}.
  The straightforward computation yields
  \[ s_{10}   = \underline{1.2}573432291, \quad
     s_{100}  = \underline{1.29903}15516, \quad
     s_{200} = \underline{1.29903810}41.
  \]
  This means the following accuracies:
  \[
  \acc(s_{10}) = 1.5, \quad
  \acc(s_{100}) = 5.3, \quad
  \acc(s_{200}) = 8.9.
  \]

  However, using only partial sums $s_{1}, s_{2}, \ldots, s_{25}$, one can 
  compute the triangular array of the~elements $\Tnw nm$ for~$1\leq n+2m\leq 
25$, $0\leq m\leq 12$.
  For instance, we have
  \[
    \Tnw{1}{10} = \underline{1.29903810}82, \quad
    \Tnw{1}{11} = \underline{1.29903810}60, \quad
    \Tnw{1}{12} = \underline{1.2990381057},
  \]
  which yields the following accuracies:
  \[
    \acc(\Tnw{1}{10}) = 7.0, \quad
    \acc(\Tnw{1}{11}) = 7.9, \quad
    \acc(\Tnw{1}{12}) = 8.7.
  \]
  The following is worth remarking: $1^\circ$ for the fixed subscript $n$, the quantities 
  $\Tnw nm$ (for consecutive values of $m$) are of better and better accuracy 
  (namely, every next element has, more or less, one exact decimal digit more);
  $2^\circ$ since Assumption~\ref{asmpt} is satisfied, this is in agreement
  with~Theorem~\ref{T:M:a}, providing that
  \[
    \lim_{n\to\infty} \frac{\Tnw nm - s}{s_n - s} = 0,
  \]
  which means that each column of the array $\Tnw nm$ converges to $s$ faster than the sequence of partial sums $s_n$;
  $3^\circ$ the accuracy of the quantities $\Tnw nm$ shows the
  faster and faster convergence for consecutive columns of the table.
\end{example}

\begin{example}\label{EX:logconv1}
  Let us consider the complex vectors of~parameters:
  $\balpha=(1.7+2.5i, 1.5+2.0i)$, $\bbeta=(1.3-3.0i, 3.2-4.0i)$.
  The hypergeometric series
  \[ \texthyper{3}{2}{\alpha_1,\alpha_2,1}{\beta_1,\beta_2}{1}
  \approx 0.7808031959823745-0.2060305207425406 i \]
  is extremely slowly convergent. Indeed, the
  partial sums $s_{n}$ for $n=10^3, 10^5, 10^6$ give accuracy of about
  $2.3$, $2.9$ and $3.2$ exact significant decimal digits, respectively.
  In contrast, using only partial sums $s_1, s_2, \ldots, s_{15}$, the 
  transformation $\Snw$ gives the quantities $\Tnw nm$ being very good 
  approximations of the limit; see
  Table~\ref{TAB:logconv1}.
  The reason for this is that the sufficient condition in~Theorem~\ref{T} is
  arguably satisfied. Indeed, the straightforward computation yields that
  numerical values of
  \[
    \sum_{k=0}^{mp-1} \frac{M^{(m)}_{k+1}(n)}{M^{(m)}(n)}
    \cdot \frac{a_{n+k}}{r_n}
  \]
  approach $1.0$ when $n\to\infty$.
  %one can observe this for $n=1,2,\ldots,10$.
  Moreover, these values are closer and closer to $1.0$ for bigger and bigger
  values of $m$. We believe this explains the convergence acceleration performed
  by $\Snw$ transformation; cf. Table~\ref{TAB:logconv1}.

  \begin{table}[htb]
    \caption{\label{TAB:logconv1} Values of~$\acc(\Tnw nm)$ for the
      hypergeometric series in Example~\ref{EX:logconv1}.}
    \center
    \begin{tabular}{c|rrrrrrrr}
      $n\diagdown m$ & \multicolumn{1}{c}{$ 0$}  & \multicolumn{1}{c}{$ 1$}  &
      \multicolumn{1}{c}{$ 2$}  & \multicolumn{1}{c}{$ 3$}  &
      \multicolumn{1}{c}{$ 4$}  & \multicolumn{1}{c}{$ 5$}  &
      \multicolumn{1}{c}{$ 6$}  & \multicolumn{1}{c}{$ 7$} \\\hline
      $ 1$ & $ 0.4$ & $ 2.6$ & $ 4.6$ & $ 6.6$ & $ 8.5$ & $10.3$ & $12.2$ &
      $14.1$\\
      $ 2$ & $ 0.7$ & $ 3.0$ & $ 5.1$ & $ 7.0$ & $ 9.0$ & $10.9$ & $12.7$\\
      $ 3$ & $ 0.9$ & $ 3.4$ & $ 5.5$ & $ 7.5$ & $ 9.4$ & $11.3$ & $13.2$\\
      $ 4$ & $ 1.1$ & $ 3.6$ & $ 5.8$ & $ 7.8$ & $ 9.8$ & $11.7$\\
      $ 5$ & $ 1.2$ & $ 3.9$ & $ 6.1$ & $ 8.2$ & $10.2$ & $12.1$\\
      $ 6$ & $ 1.3$ & $ 4.0$ & $ 6.3$ & $ 8.5$ & $10.5$\\
      $ 7$ & $ 1.3$ & $ 4.2$ & $ 6.5$ & $ 8.7$ & $10.8$\\
      $ 8$ & $ 1.4$ & $ 4.3$ & $ 6.7$ & $ 9.0$\\
      $ 9$ & $ 1.4$ & $ 4.4$ & $ 6.9$ & $ 9.2$\\
      $10$ & $ 1.5$ & $ 4.5$ & $ 7.1$\\
      $11$ & $ 1.5$ & $ 4.6$ & $ 7.2$\\
      $12$ & $ 1.5$ & $ 4.7$\\
      $13$ & $ 1.5$ & $ 4.8$\\
      $14$ & $ 1.6$\\
      $15$ & $ 1.6$
  \end{tabular}
  \end{table}
\end{example}
%
% \begin{example}\label{EX:logconv2}
%
% \end{example}

\section{Further problems}\label{S:problems}
Let us remark that the statement in~Theorem~\ref{T:M:a} does not consider the 
case of $x=1$ in the series~\eqref{E:hser}.
This case leads to the logarithmic convergence which is usually the most 
difficult to sum.
Although we cannot use Theorem~\ref{T:M:a} in order to prove the convergence
acceleration for the $\Snw$ transformation, one can always try to check
if~condition~\eqref{E:T:cond} is satisfied.
This is exactly what was depicted in Example~\ref{EX:logconv1}.
We strongly believe that the condition~\eqref{E:T:cond} is fulfilled for all the
classes of logarithmically convergent series satisfying Assumption \ref{asmpt}.
However, we do not know how to prove it in a~general way.
Let us remark that in this~case, the polynomials $M^{(m)}(n)$,
given by~\eqref{E:M}, are no longer of the~degree~$mp^2$;
cf.~\eqref{E:sum_cj}.
Therefore, the proof of Theorem~\ref{T:M:a} needs to be different in such a~case.
However, one can try to check if the condition~\eqref{E:T:cond} is satisfied, provided
that the~terms and~remainders of the~series~\eqref{E:hser} have, at~least
formally, the following asymptotic representation:
\begin{equation*}
  \frac{a_{n+1}}{a_n} \sim 1 + \frac{{b_1}}{n} + \frac{{b_2}}{n^2} + 
\ldots, \qquad
  \frac{r_{n+1}}{r_n} \sim 1 + \frac{{d_1}}{n} + \frac{{d_2}}{n^2} + 
\ldots.
\end{equation*}
Hence, one can check that
\begin{alignat*}{2}
  {b_1} &= \sum_{j=1}^p\alpha_j-\sum_{j=1}^p\beta_j,& \qquad {b_2} &= 
\frac12\left( b_1^2 - \sum_{j=1}^p\alpha_j^2+\sum_{j=1}^p\beta_j^2 \right),\\
  {d_1} &= b_1+1,& \qquad {d_2} &= (b_1^2+b_1 b_2+b_1+b_2)/b_1.
\end{alignat*}
Thus, using the fact that
\[
  \frac{a_{n+k}}{r_n} = \frac{r_{n+1}}{r_n} \frac{r_{n+2}}{r_{n+1}}
  \cdots \frac{r_{n+k}}{r_{n+k-1}} \left(1-\frac{r_{n+k+1}}{r_{n+k}}\right),
\]
one can obtain, after some algebra, that
\[ \frac{1}{r_n} \sum_{k=0}^{mp-1} \frac{M^{(m)}_{k+1}(n)}{M^{(m)}(n)} a_{n+k}
\sim 1 + \frac{\pi_2}{n^2} + \frac{\pi_3}{n^3} + \ldots. \]
This compared with~condition~\eqref{E:T:cond} in~Theorem~\ref{T:M:a} explains
the acceleration~of~the convergence obtained in Example~\ref{EX:logconv1}.
But still the~statement in~Theorem~\ref{T} is an~open problem in the case
of a~general form of the~series~\eqref{E:hser} with~$x=1$.

Maybe even more interesting is to explain why we observe that the
quantities $\Tnw nm$ give better and~better approximation of the limit of the
series, for bigger and bigger values of~$m$; cf.~Tables~\ref{TAB:ex:alt}
and~\ref{TAB:logconv1}.
We believe that the main reason is~the relationship~\eqref{E:Lm:sat} 
satisfied by the operators~$\opL m_n$ defining~$\Snw^{(m)}$ transformation.
However, the proof of
\[ \lim_{n\to\infty} \frac{\Tnw n{m+1}-s}{\Tnw nm-s} = 0, \qquad m\in\fN, \]
seems to be very difficult not only in the case of logarithmic convergence,
but also for a~linearly convergent series~\eqref{E:hser}.
We leave it for the further research.

\section*{Acknowledgements}
The authors would like to thank the anonymous referees for valuable remarks and suggestions,
as well as for pointing out many interesting possibilities
of future research in the area of series acceleration. 
We would also like to express our gratitude to Prof.~E.~J.~Weniger for his interesting
comments on our previous research and for stimulating us to further work during the 
conference \textit{Approximation and extrapolation of convergent and divergent sequences and series}
(Luminy, 2009).

%ENDOFTEXT - do not remove this comment

\bibliographystyle{abbrv}
\bibliography{bibliography}
\end{document}